\documentclass{article}

\usepackage[english]{babel}

\usepackage[letterpaper,top=2cm,bottom=2cm,left=3cm,right=3cm,marginparwidth=1.75cm]{geometry}

\usepackage{amsmath}
\usepackage{amssymb}
\usepackage{amsthm}
\usepackage{graphicx}
\usepackage{comment}
\usepackage[colorlinks=true, allcolors=blue]{hyperref}
\newtheorem{theorem}{Theorem}
\newtheorem{lemma}{Lemma}

\title{A note on the convergence of the monotone inclusion version of the primal-dual hybrid gradient algorithm}

\author{Levon Nurbekyan\\
Department of Mathematics, Emory University\\
\href{mailto:lnurbek@emory.edu}{lnurbek@emory.edu}}

\begin{document}
\maketitle

\begin{abstract}
The note contains a direct extension of the convergence proof of the primal-dual hybrid gradient (PDHG) algorithm in~\cite{chambolle16ergodic} to the case of monotone inclusions.
\end{abstract}    

\section{Introduction}

Assume that $\mathcal{H}_1,\mathcal{H}_2$ are Hilbert spaces, and $A:\mathcal{H}_1 \to 2^{\mathcal{H}_1}$, $B:\mathcal{H}_2 \to 2^{\mathcal{H}_2}$ are maximally monotone maps. Furthermore, assume that $C:\mathcal{H}_1 \to \mathcal{H}_2$ is a non-zero bounded linear operator, and consider the following pair of primal-dual monotone inclusions
\begin{equation}\label{eq:main}
    \begin{split}
        \text{find}~x\in \mathcal{H}_1~\text{s.t.}~0\in Ax+C^*(B(Cx))&\quad \text{(P)}\\
        \text{find}~y\in \mathcal{H}_2~\text{s.t.}~y\in B(Cx),~-C^*y\in Ax,~\text{for some}~x\in \mathcal{H}_1&\quad \text{(D)} 
    \end{split}
\end{equation}

When $A,B$ are subdifferential maps of proper convex lower semicontinuous functions, this previous problem reduces to a pair of primal-dual convex programs or a convex-concave saddle point problem. More specifically, if $A=\partial f_1$, $B=\partial f_2$ for $f_1 : \mathcal{H}_1 \to \overline{\mathbb{R}}$, $f_2 : \mathcal{H}_2 \to \overline{\mathbb{R}}$ then~\eqref{eq:main} is equivalent to
\begin{equation}\label{eq:main_functions}
\begin{split}
    \inf_{x \in \mathcal{H}_1} \left\{f_1(x)+f_2(Cx)\right\}=&\inf_{x \in \mathcal{H}_1} \sup_{y \in \mathcal{H}_2} \left\{f_1(x)+\langle C x, y \rangle - f_2(y) \right\}\\
    =&\sup_{y \in \mathcal{H}_2} \left\{ -f_1^*(-C^* y)-f_2(y)\right\}.
\end{split}
\end{equation}

In~\cite{chambolle11firstorder,chambolle16ergodic}, the authors introduced a first-order primal-dual splitting scheme for solving~\eqref{eq:main_functions}, which in its simplest form reads as
\begin{equation}\label{eq:pdhg_functions}
    \begin{cases}
        x^{n+1}=\underset{x\in \mathcal{H}_1}{\operatorname{argmin}}~f_1(x)+\langle Cx,y^n \rangle+\frac{\|x-x^n\|^2}{2\tau },\\
        \Tilde{x}^{n+1}=2x^{n+1}-x^n,\\
        y^{n+1}=\underset{y\in \mathcal{H}_2}{\operatorname{argmax}}~\langle C\Tilde{x}^{n+1},y \rangle-f_2(y)-\frac{\|y-y^n\|^2}{2\sigma},
    \end{cases}
\end{equation}
where $\tau,\sigma>0$. The main results in~\cite{chambolle11firstorder,chambolle16ergodic} provide convergence of ergodic sequences
\begin{equation}\label{eq:ergodic}
    X^N=\frac{1}{N}\sum_{n=1}^N x_i,\quad Y^N=\frac{1}{N}\sum_{n=1}^N y_i,
\end{equation}
under the assumption
\begin{equation}\label{eq:timesteps}
    \tau \sigma <\frac{1}{\|C\|^2}.
\end{equation}

In~\cite{vu13}, the author considers a more general version of~\eqref{eq:main} and introduces a splitting scheme, which in its simplest form reads as
\begin{equation}\label{eq:pdhg_inclusions}
    \begin{cases}
    x^{n+1}=(I+\tau A)^{-1} \left( x^n-\tau C^* y^n\right),\\
    \Tilde{x}^{n+1}=2x^{n+1}-x^n,\\
    y^{n+1}=(I+\sigma B^{-1})^{-1}\left(y^n+\sigma C \Tilde{x}^{n+1} \right).
    \end{cases}
\end{equation}
Using techniques different from the ones in~\cite{chambolle11firstorder,chambolle16ergodic}, the author in~\cite{vu13} proves the convergence of the iterates in~\eqref{eq:pdhg_inclusions} to the solution of~\eqref{eq:main} under the same assumption~\eqref{eq:timesteps}. The key idea is to rewrite~\eqref{eq:pdhg_inclusions} in the form of a forward-backward splitting algorithm analyzed in~\cite{combettes04averaged}.

In this note, we provide a direct extension of the convergence proof of~\eqref{eq:pdhg_functions} in~\cite{chambolle16ergodic} for the monotone inclusion version~\eqref{eq:pdhg_inclusions}.

\section{Notation and hypotheses}

Throughout the note, we assume that $\mathcal{H}_1,\mathcal{H}_2$ are Hilbert spaces, $A,B$ are maximally monotone, and $C$ is a non-zero bounded linear operator. Furthermore, assume that $\psi_1:\mathcal{H}_1 \to \mathbb{R}$ and $\psi_2:\mathcal{H}_2 \to \mathbb{R}$ are continuously Fr\'{e}chet differentiable convex functions, and denote by
\begin{equation}\label{eq:bregman}
\begin{split}
    D_1(x,\Bar{x})=&\psi_1(x)-\psi_1(\Bar{x})-\langle \nabla \psi_1(\Bar{x}), x-\Bar{x} \rangle,\quad x,\Bar{x} \in \mathcal{H}_1,\\
    D_2(y,\Bar{y})=&\psi_2(y)-\psi_2(\Bar{y})-\langle \nabla \psi_2(\Bar{y}), y-\Bar{y} \rangle,\quad y,\Bar{y} \in \mathcal{H}_2,
\end{split}
\end{equation}
their Bregman divergences. We assume that there exists $\alpha>0$ such that
\begin{equation}\label{eq:timesteps_gen}
    D_1(x,\Bar{x})+D_2(y,\Bar{y}) - \langle C(x-\Bar{x}) , y- \Bar{y} \rangle \geq \alpha \left( \|x-\Bar{x}\|^2+\|y-\Bar{y}\|^2 \right),\quad\forall x,\Bar{x} \in \mathcal{H}_1,\quad\forall y,\Bar{y} \in \mathcal{H}_2.
\end{equation}
Taking $y=\Bar{y}$ we obtain
\begin{equation}\label{eq:psi1_strong}
    \psi_1(x)-\psi_1(\Bar{x})-\langle \nabla \psi_1(\Bar{x}), x-\Bar{x} \rangle=D_1(x,\Bar{x})\geq \alpha  \|x-\Bar{x}\|^2, \forall x,\Bar{x} \in \mathcal{H}_1,
\end{equation}
which means that $\psi_1$ is $2\alpha$-strongly convex. Similarly, we have that
\begin{equation}\label{eq:psi2_strong}
    \psi_2(y)-\psi_2(\Bar{y})-\langle \nabla \psi_2(\Bar{y}), y-\Bar{y} \rangle=D_2(y,\Bar{y})\geq \alpha  \|y-\Bar{y}\|^2, \forall y,\Bar{y} \in \mathcal{H}_2,
\end{equation}
and so $\psi_2$ is also $2\alpha$-strongly convex.
\begin{lemma}\label{lma:resolvent_gen}
    Assume that $\mathcal{H}$ is a Hilbert space, $\psi :\mathcal{H}\to \mathbb{R}$ is a continuously Fr\'{e}chet differentiable strongly convex function, and $M:\mathcal{H} \to 2^\mathcal{H}$ is a maximally monotone operator. Furthermore, denote by
    \begin{equation*}
        D(x,\Bar{x})=\psi(x)-\psi(\Bar{x})-\langle \nabla \psi(\Bar{x}), x-\Bar{x} \rangle,\quad x,\Bar{x} \in \mathcal{H}.
    \end{equation*}
    Then the map
    \begin{equation*}
        Tx=\nabla_x D(x,\Bar{x})+Mx,\quad x\in \mathcal{H},
    \end{equation*}
    is surjective for all $\Bar{x}\in \mathcal{H}$.
\end{lemma}
\begin{proof}
Fix an arbitrary $\Bar{x} \in \mathcal{H}$. Since $x\mapsto D(x,\Bar{x})$ is convex and smooth~\cite[Theorem 20.25]{bauschke17convex} yields that $x\mapsto \nabla_x D(x,\Bar{x})$ is maximally monotone with a domain $\mathcal{H}$. Hence, by~\cite[Theorem 1]{rockafellar70sum} we have that $T$ is maximally monotone.

Next, let $(x_0,y_0) \in \operatorname{gra}M$. Then for every $x\in \mathcal{H}$ we have that
\begin{equation*}
    \begin{split}
        \inf \|Tx\| =& \inf \|\nabla \psi(x)-\nabla \psi(x_0)+Mx-y_0+(\nabla \psi(x_0)+y_0-\nabla \psi(\Bar{x}))\|\\
        \geq& \inf \|\nabla \psi(x)-\nabla \psi(x_0)+Mx-y_0\|-\|\nabla \psi(x_0)+y_0-\nabla \psi(\Bar{x})\|.
    \end{split}
\end{equation*}
Furthermore, the strong convexity of $\psi$ yields that
\begin{equation*}
    \begin{split}
        \langle \nabla \psi(x)-\nabla \psi(x_0)+Mx-y_0, x-x_0 \rangle \geq 2\alpha \|x-x_0\|^2,
    \end{split}
\end{equation*}
for some $\alpha>0$, and from Cauchy-Schwarz inequality we obtain that
\begin{equation*}
    \inf \|\nabla \psi(x)-\nabla \psi(x_0)+Mx-y_0\| \geq 2\alpha \|x-x_0\|,\quad \forall x\in \mathcal{H}.
\end{equation*}
Hence
\begin{equation*}
     \inf \|T(x)\| \geq 2\alpha \|x-x_0\|- \|\nabla \psi(x_0)+y_0-\nabla \psi(\Bar{x})\|,\quad \forall x\in \mathcal{H},
\end{equation*}
which implies
\begin{equation*}
    \lim_{\|x\| \to \infty} \|Tx\| =\infty,
\end{equation*}
and~\cite[Corollary 21.24]{bauschke17convex} concludes the proof.
\end{proof}

\section{The algorithm and its convergence}

Considering the following primal-dual splitting algorithm
\begin{equation}\label{eq:pdhg_inclusions_Bregman}
     \begin{cases}
    x^{n+1}=(\nabla_x D_1(\cdot,x^n)+ A)^{-1} \left( -C^* y^n\right),\\
    \Tilde{x}^{n+1}=2x^{n+1}-x^n,\\
    y^{n+1}=(\nabla_y D_2(\cdot,y^n)+\sigma B^{-1})^{-1}\left( C \Tilde{x}^{n+1} \right).
    \end{cases}
\end{equation}
This previous algorithm is an extension of~\cite[Algorithm 1]{chambolle16ergodic}, where the subdifferential maps are replaced by general maximally monotone maps. When
\begin{equation*}
    \psi_1(x)=\frac{\|x\|^2}{2\tau}, \quad \psi_2(y)=\frac{\|y\|^2}{2\sigma},\quad x\in \mathcal{H}_1,~y \in \mathcal{H}_2,
\end{equation*}
we obtain
\begin{equation*}
    D_1(x,\Bar{x})=\frac{\|x-\Bar{x}\|^2}{2\tau},\quad D_2(y,\Bar{y})=\frac{\|y-\Bar{y}\|^2}{2\sigma},
\end{equation*}
and~\eqref{eq:pdhg_inclusions_Bregman} reduces to~\eqref{eq:pdhg_inclusions}. Moreover the existence of an $\alpha>0$ such that~\eqref{eq:timesteps_gen} holds is equivalent to~\eqref{eq:timesteps}.

Furthermore,~Lemma~\ref{lma:resolvent_gen} guarantees that all steps in~\eqref{eq:pdhg_inclusions_Bregman} are well defined, and the algorithm will not halt.

\begin{theorem}\label{thm:main}
    Assume that~\eqref{eq:main} admits a solution $(x^*,y^*)\in \mathcal{H}_1 \times \mathcal{H}_2$, and $(x^n,\Tilde{x}^n,y^n)$ are generated by~\eqref{eq:pdhg_inclusions_Bregman} with arbitrary initial points $(x^0,\Tilde{x}^0,y^0)\in \mathcal{H}_1\times \mathcal{H}_1 \times \mathcal{H}_2$. Then the ergodic sequence~$\{(X_N,Y_N)\}$ defined in~\eqref{eq:ergodic} is bounded, and all its weak limits are solutions of~\eqref{eq:main}.
\end{theorem}
\begin{proof}
    We introduce the following function
    \begin{equation}\label{eq:L}
    \begin{split}
        \mathcal{L}(x,\zeta;y,\eta)=&\sup_{(u,v)\in Ax\times  B^{-1}y}\langle x-\zeta, -u-C^*\eta \rangle+\langle C \zeta-v, y-\eta \rangle\\
        =&\sup_{(u,v)\in Ax\times  B^{-1}y} \langle \zeta-x, u \rangle+\langle \eta-y, v  \rangle - \langle C x,\eta \rangle + \langle C \zeta, y \rangle, 
    \end{split}
    \end{equation}
    where we set the supremum of an empty set to be $-\infty$. As pointed out in~\cite{chambolle16ergodic}, the basic building block of~\eqref{eq:pdhg_inclusions_Bregman} is the iteration
    \begin{equation}\label{eq:basic_iter}
        \begin{cases}
            \hat{x}=(\nabla_x D_1(\cdot,\Bar{x})+ A)^{-1} \left( -C^* \Tilde{y}\right),\\
            \hat{y}=(\nabla_y D_2(\cdot,\Bar{y})+\sigma B^{-1})^{-1}\left( C \Tilde{x} \right),
        \end{cases}
    \end{equation}
    for suitable choices of $\Bar{x},\hat{x},\Tilde{x}$ and $\Bar{y},\hat{y},\Tilde{y}$. In an expanded form,~\eqref{eq:basic_iter} can be written as
    \begin{equation}\label{eq:basic_iter_expanded}
        \begin{cases}
            \nabla_x D_1(\hat{x},\Bar{x})+\hat{u}=-C^* \Tilde{y},\\
            \nabla_y D_2(\hat{y},\Bar{y})+\hat{v}=C \Tilde{x},
        \end{cases}
    \end{equation}
    where $(\hat{u},\hat{v})\in A \hat{x} \times B^{-1} \hat{y}$. Thus, we first obtain estimates for the general iteration~\eqref{eq:basic_iter_expanded} and then apply them to~\eqref{eq:pdhg_inclusions_Bregman}.

    Let~\eqref{eq:basic_iter_expanded} hold, and $(x,y)\in \mathcal{H}_1 \times \mathcal{H}_2$, $(u,v) \in Ax \times B^{-1}y$ be arbitrary. Then by the monotonicity of $A$ and~\eqref{eq:basic_iter_expanded} we have that
    \begin{equation}\label{eq:aux1}
        \begin{split}
            \langle u, x-\hat{x} \rangle \geq & \langle \hat{u},x-\hat{x} \rangle = \langle -C^* \Tilde{y}-\nabla_x D_1(\hat{x},\Bar{x}),x-\hat{x} \rangle\\
            =&\langle -C^* \Tilde{y},x-\hat{x} \rangle+D_1(\hat{x},\Bar{x})+D_1(x,\hat{x})-D_1(x,\Bar{x}),
        \end{split}
    \end{equation}
    where we also used the identity
    \begin{equation*}
        \langle-\nabla_x D_1(\hat{x},\Bar{x}),x-\hat{x} \rangle=D_1(\hat{x},\Bar{x})+D_1(x,\hat{x})-D_1(x,\Bar{x}).
    \end{equation*}
    Similarly, using the monotonicity of $B^{-1}$ we obtain
    \begin{equation}\label{eq:aux2}
        \begin{split}
            \langle v, y-\hat{y} \rangle \geq & \langle \hat{v},y-\hat{y} \rangle = \langle C \Tilde{x}-\nabla_x D_2(\hat{x},\Bar{x}),x-\hat{x} \rangle\\
            =&\langle C \Tilde{x},y-\hat{y} \rangle+D_2(\hat{y},\Bar{y})+D_2(y,\hat{y})-D_2(y,\Bar{y}).
        \end{split}
    \end{equation}
    Combining~\eqref{eq:aux1},~\eqref{eq:aux2}, we obtain
    \begin{equation*}
        \begin{split}
            &D_1(x,\Bar{x})-D_1(\hat{x},\Bar{x})-D_1(x,\hat{x})+D_2(y,\Bar{y})-D_2(\hat{y},\Bar{y})-D_2(y,\hat{y})\\
            \geq& \langle x-\hat{x}, -u-C^* \Tilde{y} \rangle+ \langle C \Tilde{x}-v,y-\hat{y} \rangle\\
            =&\langle x-\hat{x}, -u-C^* \hat{y} \rangle+ \langle C \hat{x}-v,y-\hat{y} \rangle+\langle C(x-\hat{x}), \hat{y}-\Tilde{y} \rangle+ \langle C(\Tilde{x}-\hat{x}), y-\hat{y} \rangle.
        \end{split}
    \end{equation*}
    Since $(u,v) \in A x \times B^{-1}y$ are arbitrary, we obtain that
    \begin{equation}\label{eq:main_estimate}
    \begin{split}
        \mathcal{L}(x,\hat{x};y,\hat{y}) \leq & D_1(x,\Bar{x})-D_1(\hat{x},\Bar{x})-D_1(x,\hat{x})+D_2(y,\Bar{y})-D_2(\hat{y},\Bar{y})-D_2(y,\hat{y})\\
        &+\langle C(x-\hat{x}), \Tilde{y}-\hat{y} \rangle+ \langle C(\Tilde{x}-\hat{x}), \hat{y}-y \rangle,\quad \forall x\in \mathcal{H}_1,~y\in \mathcal{H}_2.
    \end{split}
    \end{equation}
    As in~\cite{chambolle16ergodic}, this previous inequality is the key inequality in the proof. Indeed,~\eqref{eq:pdhg_inclusions_Bregman} corresponds to choosing
    \begin{equation*}
        \hat{x}=x^{n+1},~\Bar{x}=x^n,~\Tilde{x}^{n+1}=2x^{n+1}-x^n,~\hat{y}=y^{n+1},~\Bar{y}=y^n,~\Tilde{y}=y^n,
    \end{equation*}
    in~\eqref{eq:basic_iter}, and so~\eqref{eq:main_estimate} yields
    \begin{equation*}
        \begin{split}
        \mathcal{L}(x,x^{n+1};y,y^{n+1}) \leq & \left\{ D_1(x,x^n)+D_2(y,y^n)-\langle C(x-x^n),y-y^n \rangle\right\}\\
        &-\left\{ D_1(x,x^{n+1})+D_2(y,y^{n+1})-\langle C(x-x^{n+1}),y-y^{n+1} \rangle\right\}\\
        &-\left\{ D_1(x^{n+1},x^{n})+D_2(y^{n+1},y^{n})-\langle C(x^{n+1}-x^{n}),y^{n+1}-y^{n} \rangle\right\}.
    \end{split}
    \end{equation*}
    Hence, by the convexity of $(\zeta,\eta) \mapsto \mathcal{L}(x,\zeta;y,\eta)$, we obtain
    \begin{equation}\label{eq:aux3}
        \begin{split}
          N \mathcal{L}(x,X^N;y,Y^N) \leq &  \sum_{n=1}^N \mathcal{L}(x,x^n;y,y^n)\\
            \leq & \left\{ D_1(x,x^0)+D_2(y,y^0)-\langle C(x-x^0),y-y^0 \rangle\right\}\\
            & - \left\{ D_1(x,x^{N})+D_2(y,y^{N})-\langle C(x-x^{N}),y-y^{N} \rangle\right\}\\
            &-\sum_{n=1}^N \left\{ D_1(x^{n},x^{n-1})+D_2(y^{n},y^{n-1})-\langle C(x^{n}-x^{n-1}),y^{n}-y^{n-1} \rangle\right\},
        \end{split}
    \end{equation}
    for all $x\in \mathcal{H}_1$, $y\in \mathcal{H}_2$, and $N\in \mathbb{N}$. Note that~\eqref{eq:timesteps_gen} guarantees that the expressions in the curly brackets are nonnegative.
    
    Recall that $(x^*,y^*)$ is a solution of~\eqref{eq:main}, and so
    \begin{equation}\label{eq:x*y*sol}
        -C^*y^* \in Ax^*,\quad Cx^* \in B^{-1} y^*.
    \end{equation}
    But then by the definition of $\mathcal{L}$ we have that
    \begin{equation*}
        \mathcal{L}(x^*,\zeta;y^*,\eta)\geq \langle x^*-\zeta,C^* y^*-C^*\eta \rangle+\langle C \zeta-Cx^*, y^*-\eta \rangle=0,\quad \forall \zeta\in \mathcal{H}_1,~\forall \eta\in \mathcal{H}_2.
    \end{equation*}
    In particular, we have that
    \begin{equation}\label{eq:pd_gap>=0}
    \mathcal{L}(x^*,X^N;y^*,Y^N)\geq 0,    
    \end{equation}
    and~\eqref{eq:aux3} yields that
    \begin{equation*}
        D_1(x^*,x^{N})+D_2(y^*,y^{N})-\langle C(x^*-x^{N}),y^*-y^{N} \rangle\leq D_1(x^*,x^0)+D_2(y^*,y^0)-\langle C(x^*-x^0),y-y^0 \rangle,
    \end{equation*}
    and~\eqref{eq:timesteps_gen} implies that
    \begin{equation*}
        \|x^N-x^*\|^2+\|y^N-y^*\|^2 \leq \frac{D_1(x^*,x^0)+D_2(y^*,y^0)-\langle C(x^*-x^0),y-y^0 \rangle}{\alpha},\quad \forall N\in \mathbb{N}.
    \end{equation*}
    Therefore, $\{(x^n,y^n)\}$ is a bounded sequence, and the convexity of the norm yields the boundedness of the ergodic sequence with the same bounds; that is,
    \begin{equation*}
        \|X^N-x^*\|^2+\|Y^N-y^*\|^2 \leq \frac{D_1(x^*,x^0)+D_2(y^*,y^0)-\langle C(x^*-x^0),y-y^0 \rangle}{\alpha},\quad \forall N\in \mathbb{N}.
    \end{equation*}

    Assume that $(X,Y)$ is a weak (subsequential) limit of $\{(X_N,Y_N)\}$. Invoking~\eqref{eq:aux3} again, we obtain
    \begin{equation}\label{eq:aux4}
        \mathcal{L}(x,X^N;y,Y^N)\leq \frac{ D_1(x,x^0)+D_2(y,y^0)-\langle C(x-x^0),y-y^0 \rangle}{N},
    \end{equation}
    for all $x\in \mathcal{H}_1$, $y\in \mathcal{H}_2$, and $N\in \mathbb{N}$. Let $(u,v)\in Ax \times B^{-1}y$ be arbitrary. Then we have that
    \begin{equation*}
        \begin{split}
            \langle X^N-x, u \rangle+\langle Y^N- y,v \rangle-\langle C x,Y^N \rangle+\langle C X^N, y \rangle \leq \mathcal{L}(x,X^N;y,Y^N),
        \end{split}
    \end{equation*}
    and so the weak convergence and~\eqref{eq:aux4} yield
    \begin{equation*}
        \begin{split}
        &\langle X-x, u \rangle+\langle Y- y,v \rangle-\langle C x,Y \rangle+\langle C X, y \rangle\\
          =&  \lim_{N \to \infty}\langle X^N-x, u \rangle+\langle Y^N- y,v \rangle-\langle C x,Y^N \rangle+\langle C X^N, y \rangle\\
          \leq & \liminf_{N \to \infty} \mathcal{L}(x,X^N;y,Y^N) \leq 0.
        \end{split}
    \end{equation*}
    Therefore we have that
    \begin{equation}\label{eq:L_optimal}
        \mathcal{L}(x,X;y,Y)\leq 0,\quad \forall x\in \mathcal{H}_1,~y \in \mathcal{H}_2.
    \end{equation}
    Taking $y=Y$ in~\eqref{eq:L_optimal} we obtain
    \begin{equation*}
        \langle x-X, u+C^* Y \rangle \geq 0,\quad \forall (x,u) \in \operatorname{gra} A,
    \end{equation*}
    and so maximal monotonicity of $A$ yields that
    \begin{equation}\label{eq:aux5}
        (X,-C^* Y) \in \operatorname{gra} A \Longleftrightarrow -C^* Y \in A X.
    \end{equation}
    Similarly, plugging in $x=X$ in~\eqref{eq:L_optimal} we find that
    \begin{equation*}
        \langle y-Y,v-CX \rangle \geq 0,\quad \forall (y,v)\in \operatorname{gra} B^{-1},
    \end{equation*}
    and the maximal monotonicity of $B^{-1}$ yields that
    \begin{equation}\label{eq:aux6}
        (Y,CX) \in \operatorname{gra} B^{-1} \Longleftrightarrow Y \in B(CX).
    \end{equation}
    Combining~\eqref{eq:aux5} and \eqref{eq:aux6} we obtain that $(X,Y)$ is a solution of~\eqref{eq:main}.
\end{proof}

\bibliographystyle{plain}
\bibliography{pdhg}

\end{document}